\documentclass[12pt,A4paper]{article}

\usepackage[left=1in, right=1in, top=1in, bottom=1in]{geometry}
\usepackage{labelfig}
\usepackage{epsfig}
\usepackage{epstopdf}
\usepackage{soul}
\usepackage{xfrac}

\usepackage[percent]{overpic}

\usepackage{color}
\usepackage{amsthm,amsmath,amssymb}
\usepackage{booktabs}
\usepackage{mathpazo}
\usepackage{microtype}
\usepackage{overpic}
\usepackage{bm}
\usepackage{sectsty}
\usepackage[
	pdftitle={PDFTitle},
	pdfauthor={Hugo Parlier},
	ocgcolorlinks,
	linkcolor=linkred,
	citecolor=linkred,
	urlcolor=linkblue]
{hyperref}

\definecolor{linkred}{RGB}{199,21,133}
\definecolor{linkblue}{RGB}{16, 78, 139}
\definecolor{darkgreen}{RGB}{60, 140, 60} 

\usepackage[hang,flushmargin]{footmisc}
\usepackage{enumitem}
\usepackage{titlesec}
	\titlespacing{\section}{0pt}{12pt}{0pt}
	\titlespacing{\subsection}{0pt}{6pt}{0pt}
	
\titlelabel{\thetitle.\quad}

\makeatother 

\theoremstyle{plain}
\newtheorem{theorem}{Theorem}[section]

\newtheorem{lemma}[theorem]{Lemma}
\newtheorem{corollary}[theorem]{Corollary}

\theoremstyle{definition}

\newtheorem{remark}[theorem]{Remark}

\newcommand{\sys}{{\rm sys}}
\newcommand{\Hyp}{{\mathbb H}}
\newcommand{\arcsinh}{{\,\rm arcsinh}}
\newcommand{\arccosh}{{\,\rm arccosh}}
\newcommand{\dist}{\mathrm{dist}}

\sectionfont{\large \bfseries}
\subsectionfont{\normalsize}

\long\def\symbolfootnote[#1]#2{\begingroup%
\def\thefootnote{\fnsymbol{footnote}}\footnote[#1]{#2}\endgroup}

\def\blfootnote{\xdef\@thefnmark{}\@footnotetext}

\usepackage{mathtools}

\usepackage[sort,nocompress]{cite}

\begin{document}

{\Large \bfseries Finding  the Birman-Series set}

{\large Peter Buser and Hugo Parlier\symbolfootnote[1]{\small Supported by ANR-SNF Grant number 200021E\_238147
(SUGAR).
}
}

\vspace{0.2cm}
{\bf Abstract.}
By a celebrated result of Birman and Series, the set of simple complete geodesics form a nowhere dense subset of a hyperbolic surface. We show that a segment of finite length is in a (small) neighborhood of this set if it can be "extended simply" in both directions for long enough. We deduce that, for large enough genus $g$, there exist surfaces with a Birman-Series set that is roughly $\frac{1}{g^{\sfrac{2}{3}}}$-dense.

\section{Introduction} \label{sec:intro}

The study of closed geodesics on hyperbolic surfaces takes many forms, and relates topics such as dynamics, number theory and geometric group theory, while being fundamentally geometric objects as orbits in the geodesic flow. From a topological viewpoint, simple closed curves play a fundamental role in the understanding of mapping class groups, while their geodesic counterparts help decipher moduli space and friends.

Birman and Series \cite{Birman-Series} discovered a striking geometric property for the set of all simple closed geodesics, and more generally of all simple complete geodesics: they are nowhere dense on a surface and have Hausdorff dimension $1$. This so-called Birman-Series set is in striking contrast to the set of {\it all} closed geodesics, which is dense on both the surface and on its unit tangent bundle. One of the main ingredients which makes the Birman-Series theorem work is that there are "few" simple geodesics, who grow polynomially in length versus exponentially for closed geodesics in general \cite{Huber}. The polynomial growth phenomenon was later famously shown to have an asymptotic growth formula by Mirzakhani \cite{Mirzakhani2008}. So, by any measure, simple geodesics are rare, and, following the philosophy of what is seldom is wonderful, we prove the following.

\begin{theorem}\label{thm:main}
Consider a hyperbolic surface $X$ and $\frac{1}{10}>\varepsilon >0$. If $p\in X$ is the midpoint of a simple geodesic arc of length 
$$
2 \log\left(\frac{1}{\varepsilon} \right) +20
$$
then $p$ is at distance at most $\varepsilon$ from a geodesic lamination. 
\end{theorem}

The non-density of the Birman-Series set means that, for small enough $\varepsilon$, there are regions of the surface in which all points are $\varepsilon$ away from all complete simple geodesics. The above result says that if a point lies in the middle of a long enough simple geodesic arc, it is in the $\varepsilon$-neighborhood of the Birman-Series set of $X$. The result applies just as well to geodesic segments: any small geodesic segment which is the subset of a quantifiably long embedded geodesic segment  is close to a geodesic lamination. This second way of framing the result can be thought of as a result about the unit tangent bundle. We preferred the above formulation for simplicity. 

We have made little effort to optimize the constants in Theorem \ref{thm:main}, but they have the advantage of being very concrete and rather simple. In previous work \cite{Buser-Parlier}, we proved lower bounds, which only depend on the topology of the surface, on the radius of disks of a size which are disjoint from the Birman-Series set, but the constants, while explicit, are somewhat more involved. As a consequence of Theorem \ref{thm:main}, we obtain the following upper bounds.
\begin{corollary}\label{cor:bs} 
There exist a sequence of surfaces of genus $g_k\to \infty$ with a Birman-Series set that is roughly $\frac{1}{g_{k}^{\sfrac{2}{3}}}$-dense. 
\end{corollary}
The statement above is a slight simplification of Corollary \ref{cor:dense} in the text, which provides an exact estimate. The corollary is a direct application of Theorem \ref{thm:main} to surfaces with large systole, with growth on the order of $\frac{4}{3} \log(g)$ first exhibited in \cite{Buser-Sarnak}.

The idea of the proof of Theorem \ref{thm:main} is as follows. Given a random point in the unit tangent bundle, exponentiating far enough in both directions will produce a non-simple geodesic arc. Of course, this isn't the case if the original tangent vector belongs to the Birman-Series set, but this a rare occurrence and constitutes a measure zero set in the unit tangent bundle. We examine the geometry at the first points of self-intersection, and build the lamination based on the local geometry. This phenomenon is very similar to one used to prove Luo-Tan's beautiful identity \cite{Luo-Tan}, although we exploit it quite differently.

Finally, we note that the original Birman-Series result has been generalized in different contexts, namely by Sapir to geodesics that have bounded intersection in terms of their length \cite{Sapir}, and by Lenzhen-Souto \cite{LenzhenSouto} to complete non-filling geodesics. In principle, the ideas from Theorem \ref{thm:main} could also be applied in different contexts, but additional difficulties appear as soon as you allow self-intersections. It definitely provides an interesting direction in which to explore. 
 
\noindent {\bf Acknowledgements.}

The authors are grateful to the Bernoulli Center at the EPFL for hosting inspiring events where some of our mathematical discussions took place.  Finally, as it is now standard, we declare that we didn't use AI in any way or form, despite the proofs involving a somewhat detailed case by case analysis which we did the old-fashioned way. 

\section{Preliminaries}

\subsection{Definitions and notations}
Throughout $X$ is an orientable surface, endowed with a geodesically complete hyperbolic metric of finite area. While some of the results make sense in a more general context, this is the framework of interest. However, Corollary \ref{cor:bs} exhibits a sequence of closed surfaces of genus $g \geq 2$. 

We are interested in the Birman-Series set of such a hyperbolic surface, that is the set of all simple complete geodesics on $X$. Birman and Series \cite{Birman-Series} showed that the set of all simple complete geodesics is nowhere dense on $X$  and has Hausdorff dimension $1$. This result was a cornerstone for many subsequent results on the study of simple closed geodesics on a surface (which are of course contained in the Birman-Series set). 

\subsection{Setup}

Here we present three lemmas about broken geodesics in the hyperbolic plane $\Hyp$.  We look at connected piecewise geodesic paths which we extend in both directions. The lemmas provide a sufficient angle-segment-length condition for them to be embedded and to be close to a given geodesic. By angles of a piecewise geodesic we mean the smallest angles in the intersection points, that is the complementary angle to the "turning" oriented angle at singular points. So angles at least $\theta_0$ means that the (oriented) angle has (absolute) value at most $\pi - \theta_0$ (see Figure \ref{fig:broken} for a depiction of the turning angles). 

 \begin{figure}[h]
\begin{center}
\leavevmode
\SetLabels
\endSetLabels
\AffixLabels{
\epsfig{file =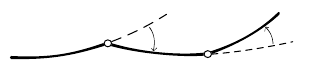,width=6cm,angle=0}}
\vspace{-24pt}
\end{center}
\caption{Turning angles of a broken geodesic} \label{fig:broken}
\end{figure}

The first lemma applies to both finite and infinite length paths. 

\begin{lemma}\label{lem:broken}
Let $\theta_0 \in ]0,\pi[$. Then any piecewise geodesic with angles at least $\theta_0$ and lengths of segments at least 
\begin{equation*}
\ell(\theta_0) = 2 \arccosh\left( \frac{1}{\sin\left(\frac{\theta_0}{2}\right)}\right)
\end{equation*}
remains embedded. If it is infinite in both directions, then it has two distinct endpoints on $\partial \Hyp$, except in the case where all angles equal $\theta_0$ and all segments have length exactly $\ell_0$.
\end{lemma}

\begin{proof}
Consider, for each vertex $p$ of the polygon, the geodesic in the hyperbolic plane passing through $p$ that bisects the angle, and for each side of the polygon the orthogonal geodesic through the midpoint. For neighboring pairs of these geodesics, say bisecting angle $\theta$ and length $\ell$, we have
\begin{equation*}
\cosh(\frac{\ell}{2}) \sin\left(\frac{\theta}{2}\right) \geq \cosh\left(\frac{\ell(\theta_0)}{2}\right) \sin\left(\frac{\theta_0}{2}\right) = 1.
\end{equation*}
By the right-angled triangle formula the two intersect at most at infinity and the latter only if $\theta = \theta_0$ and $\ell = \ell_0$. Hence, the bisecting geodesics separate the hyperbolic plane into pairwise disjoint strips and the lemma follows.
\end{proof}

Our proofs will ultimately need the above lemma for $\theta_0 = \frac{\pi}{8}$, so we set 
$$
\ell_0:= 2 \arccosh\left( \frac{1}{\sin\left(\frac{\pi}{16}\right)}\right) =  4.63\hdots  < 5.
$$

The second lemma approximates the distance between a point that lies on a geodesic segment between two geodesics as in Figure \ref{fig:epsilon}.

 \begin{figure}[h]
\begin{center}
\leavevmode \SetLabels
\L(.035*.30) $A$\\%
\L(.92*.30) $B$\\%
\endSetLabels
\AffixLabels{
\epsfig{file =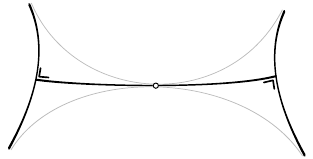,width=6cm,angle=0}}
\vspace{-24pt}
\end{center}
\caption{A point close to all geodesics going left to right} \label{fig:epsilon}
\end{figure}

\begin{lemma}\label{lem:epsilon}
Fix $\varepsilon>0$. Consider two disjoint geodesics, $A$ and $B$, and $p$ a point lying on the shortest path between $A$ and $B$ and such that $d(p,A)\geq d(p,B) \geq \ell_\varepsilon := \arcsinh \left( \frac{1}{\sinh(\varepsilon)}\right)$.

Now let $C$ be any geodesic which intersects $A$ and $B$. Then $d(p,C) < \varepsilon$. 
\end{lemma}

\begin{proof}
By a variational argument, the extremal case to consider is when $p$ lies at distance exactly $\ell_\varepsilon$ from $A$ and $B$. Consider the two limit geodesics which form an ideal quadrilateral with $A$ and $B$. By symmetry, the shortest path between them passes through $p$. Denote the length of this path $2x$. By cutting along it we obtain a quadrilateral with two ideal points and two right angles. Now by cutting them along the shortest path between $A$ and $B$, we get four isometric quadrilaterals, each with an ideal point and three right angles (so-called Lambert quadrilaterals). The proof follows from standard hyperbolic trigonometry considerations in this quadrilateral. Indeed two of its adjacent sides are of length $x$ and $\ell_\varepsilon$ and hence satisfy $\sinh(x) \sinh(\ell_\varepsilon) = 1$. It follows that $x = \varepsilon$. 
\end{proof}

\begin{figure}[h]
\begin{center}
\leavevmode \SetLabels
\L(.028*.745) $\Delta$\\%
\L(.965*.32) $\Delta$\\%
\L(.07*.95) $p^{-}$\\%
\L(.965*.05) $p^{+}$\\%
\L(.50*.55) $p$\\%
\L(.004*.522) $a$\\%
\L(.972*.56) $b$\\%
\L(.06*.30) $A$\\%
\L(.87*.70) $B$\\%
\L(.07*.56) $\ell_0$\\%
\L(.89*.45) $\ell_0$\\%
\endSetLabels
\AffixLabels{
\epsfig{file =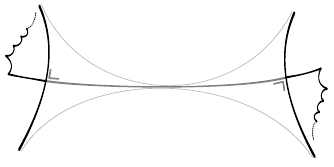,width=6.5cm,angle=0}}
\vspace{-24pt}
\end{center}
\caption{The configuration in Lemma \ref{lem:outside}} \label{fig:outside}
\end{figure}

\begin{lemma}\label{lem:outside}
Let $\Delta$ be a broken geodesic in the hyperbolic plane with ideal endpoints $p^{-}, p^{+} \in \partial \Hyp$, with all segments of lengths $\geq \ell_0$ and all angles $\geq \theta_0$. Let $p$ be a point on a ``large'' segment $ab$ of $\Delta$  with distances $ \dist(p,a) \geq \ell_{\varepsilon}+ \ell_0$, $dist(p,b)\geq \ell_{\varepsilon}+ \ell_0$. Then $p$ is $\varepsilon$-close to the geodesic that connects $p^{-}$ with $p^{+}$.
\end{lemma}
\begin{proof}
Consider the geodesics $A, B$ that intersect $ab$ orthogonally at distances $\ell_0$ from $a$ and $b$ respectively (Figure \ref{fig:outside}). By Lemma \ref{lem:broken} the polygon that goes from $p^{-}$ along $\Delta$ to $a$, then along $ab$ to $A$ and finally along $A$ to $\partial \Hyp$ is embedded. Hence, $p^{-}$ and $p$ are separated by $A$. Similarly, $B$ separates $p$ and $p^{+}$. The lemma now follows from Lemma \ref{lem:epsilon}.
\end{proof}

\section{Proof of the main result}

\subsection{A short description of the method}
Before going into the technical difficulties, we explain the approach. We set
\begin{equation*}
R_\varepsilon =  \ell_{\varepsilon}+2\ell_0 \leq \log\left(\frac{1}{\varepsilon} \right) +10,
\end{equation*}
where $\ell_{\varepsilon}$ is from Lemma \ref{lem:epsilon}.

We start with a point in the unit tangent bundle and exponentiate in both directions at unit speed. 

By hypothesis, for time less than $R_\varepsilon$ this process produces a (finite) simple arc. We continue to exponentiate beyond this threshold length and stop only if this process produces a point of self-intersection. In fact, with the exception of a measure zero set of initial points, which we may ignore, there are first points of self intersections in either direction and we stop exponentiating exactly there.
 \begin{figure}[h]
\begin{center}
\leavevmode \SetLabels
\L(.35*.77) $\Gamma$\\%
\endSetLabels
\AffixLabels{
\epsfig{file =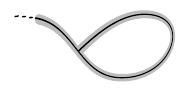,width=5cm,angle=0}}
\vspace{-24pt}
\end{center}
\caption{Finding a small surface around the exponentiated arc $\Gamma$} \label{fig:seetopomethod}
\end{figure}

For the resulting geodesic arc, $\Gamma$, there are only finitely many topological cases to consider and each of them represents a spine of either an embedded hyperbolic pair of pants or a one-holed torus. By considering the possible configurations, one arrives at the list illustrated in Figure \ref{fig:seetopo}. We now note that the third and fourth cases, albeit not isotopic on the surface, are topologically equivalent configurations.
So in fact there are only four cases to consider. We  label the cases from I to V as in the figure, noting that we will not treat Case III.

\begin{figure}[h]
\begin{center}
\leavevmode \SetLabels
\L(.105*.00) I\\%
\L(.295*.00) II\\%
\L(.50*.00) III\\%
\L(.69*.00) IV\\%
\L(.89*.00) V\\%
\endSetLabels
\AffixLabels{
\epsfig{file =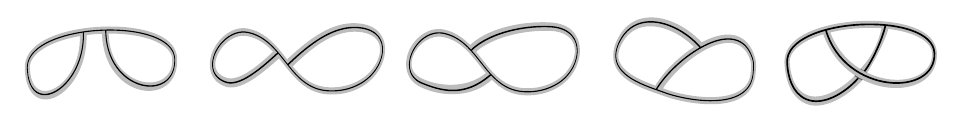,width=15cm,angle=0}}
\vspace{-12pt}
\end{center}
\caption{The topological cases for the spines: Cases III and IV are equivalent} \label{fig:seetopo}
\end{figure}

The next step is to examine the geometry of the resulting spines. In particular, we look at the angles at intersection points and the lengths of the geodesic loops inside the spine. In each possible case, we carefully construct an element of the Birman-Series set that is close to $\Gamma$. This will either be a simple closed geodesic or a bi-infinite simple complete geodesic. We call it the {\it target lamination}. In all cases, the proximity between the point on the surface and the target lamination will be deduced from Lemma \ref{lem:outside} above. 

The proof proceeds by a case analysis on the possible spines.
In each case we construct an embedded broken geodesic with uniformly bounded angles and sufficiently long segments; Lemma~\ref{lem:outside} guarantees embeddedness and
Lemma~\ref{lem:epsilon} gives the required $\varepsilon$--closeness.

\subsection{The case by case analysis}

In each of these cases, we argue using a mixture of topological and geometric arguments and provide a topological description of the simple complete geodesic that is close to the point.

\subsubsection{Case I}

We first consider the topological case as in Figure \ref{fig:seecase01}. 

\begin{figure}[h]
\begin{center}
\leavevmode
\SetLabels
\L(.355*-.25) (i)\\%
\L(.53*-.25) (ii)\\%
\L(.705*-.25) (iii)\\%
\L(.88*-.25) (iv)\\%
\endSetLabels
\AffixLabels{
\epsfig{file =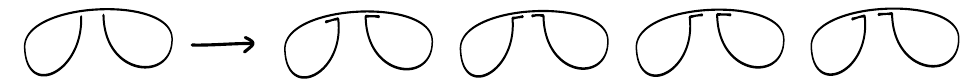,width=15cm,angle=0}}
\vspace{-12pt}
\end{center}
\caption{Case I} \label{fig:seecase01}
\end{figure}

There are 4 sub-cases to consider, depending on the angles of intersection, which we will number (from left to right) by (i) to (iv). At a point of intersection, at least one of the angles is $\geq \frac{\pi}{2}$. We indicate the greater of the two angles on the figures (for example as in Figure \ref{fig:seecase01}) by a short arc making a turn towards the greater angle. Note that the cases are not exclusive, as both angles could be exactly $\frac{\pi}{2}$. Finally note that the two right-most subcases from Figure \ref{fig:seecase01}) are symmetric, and so we will only need to treat one of them. 

{\it Case I(i).} For this configuration we will apply what we will call the {\it loop trick} which we now explain. 

\begin{figure}[b]
\begin{center}
\leavevmode \SetLabels
\L(.17*1.03) $p_2$\\%
\L(.37*.797) $\frac{\theta_2}{2}$\\%
\L(.90*.77) $\gamma_2$\\%
\endSetLabels
\AffixLabels{
\epsfig{file =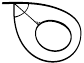,width=4cm,angle=0}}
\vspace{-24pt}
\end{center}
\caption{The loop trick} \label{fig:halfangles}
\end{figure}

\noindent {\it The loop trick.} Note that each intersection point is a basepoint of a geodesic loop. Consider one of these say $\gamma_2$ at $p_2$, with interior angle $\theta_2 \geq \frac{\pi}{2}$. This loop is (freely) homotopic to a unique simple closed geodesic (a core geodesic). Together they bound an embedded cylinder and there is a unique geodesic path inside the cylinder between the basepoint of the loop and perpendicular to the core geodesic. This path splits the angle at the basepoint into two half-angles of equal measure, both $\geq \frac{\pi}{4}$. This is illustrated in Figure \ref{fig:halfangles}.

We do this for both loops (a double-loop trick). We can now identify the target complete simple geodesic that will be $\varepsilon$ close to our point. It is the one in the homotopy class obtained by following the loop one time around and then wrapping around the core geodesic (and this on both sides). This gives a piecewise geodesic with two singular points $p_1$, $p_2$ with angles at least $\frac{\pi}{4}$: in fact the angle $\theta'$ in Figure \ref{fig:looping} is at least $\frac{\pi}{4}$ and it is just a subangle of the full angle $\theta'_2$ at the singular point $p_2$. Our target simple geodesic is the one illustrated in Figure \ref{fig:topo1}. 

\begin{figure}[h]
\rule{0pt}{12pt}
\begin{center}
\leavevmode \SetLabels
\L(.23*.64) $\theta'$\\%
\L(.15*1.03) $p_2$\\%
\L(.205*0.819) $\theta'_2$\\%
\endSetLabels
\AffixLabels{
\epsfig{file =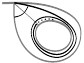,width=4cm,angle=0}}
\vspace{-12pt}
\end{center}
\caption{Wrapping around the core geodesic} \label{fig:looping}
\end{figure}

\begin{figure}[h]
\begin{center}
\leavevmode \SetLabels
\L(.390*.828) $p_1$\\%
\L(.485*.85) $p_2$\\%
\L(.33*.725) $\theta_1$\\%
\L(.565*.78) $\theta_2$\\%
\endSetLabels
\AffixLabels{
\epsfig{file =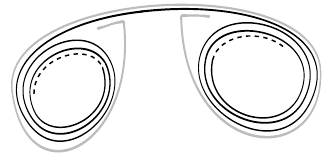,width=6cm,angle=0}}
\vspace{-12pt}
\end{center}
\caption{The target simple complete geodesic for Case I(i)} \label{fig:topo1}
\end{figure}

To show that our choice was judicious, we consider the lift of the piecewise geodesic curve in the universal cover and compare it to the corresponding lift of the target geodesic (see Figure \ref{fig:universal1}). We now observe hat we are exactly in the angle-segment-length conditions of Lemma \ref{lem:outside} and from there we are able to conclude. 

\begin{figure}[h]
\begin{center}
\leavevmode \SetLabels
\L(.43*.5) $p$\\%
\L(.170*.76) $p_1$\\%
\L(.693*1.00) $p_2$\\%
\L(.19*.49) $\theta'_1$\\%
\L(.685*.738) $\theta'_2$\\%
\endSetLabels
\AffixLabels{
\epsfig{file =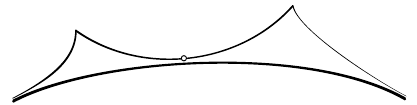,width=10cm,angle=0}}
\vspace{-24pt}
\end{center}
\caption{In the universal cover. The angles satisfy $\frac{\pi}{4} \leq \theta'_k< \theta_k$, $k=1,2$.
} \label{fig:universal1}
\end{figure}


Observe that, in this particular case, we know that the angle is at least $\frac{\pi}{4}$, but there are cases we will consider in the sequel where the weaker angle condition $\theta_k \geq \frac{\pi}{8}$ comes into play.

{\it Case I(ii).} We now consider the next case where we use another trick that we call the {\it insertion trick} which is used when certain geometric elements of the spine are too short. 

{\it The insertion trick.} Consider the bold arc of length labelled $x$ on $\Gamma$ in Figure \ref{fig:sticktrick} connecting the endpoints $p_1$, $p_2$. We denote again by $\gamma_1$, $\gamma_2$ the geodesic loops formed by $\Gamma$ at $p_1$, $p_2$, where we assume w.l.o.g.\ that $\ell(\gamma_1) \geq \ell(\gamma_2)$.

If $x\geq \ell_0$, then we don't modify anything, and, using the previous argument in the universal cover, we lift the spine to a broken geodesic consisting successively of arcs of lengths $x$ and some $R \geq 2R_{\epsilon}$ and having successive angles $\theta_1$, $\theta_2$ at the singular points. The target lamination is the exterior simple closed geodesic of the pair of pants, shown in Figure \ref{fig:sticktrick} (a).

If, however, $x < \ell_0$, then we insert an additional geodesic arc of length $\ell_0$ as in Figure \ref{fig:sticktrick} (b), going from $p_2$ to its endpoint $p'_1$ on $\gamma_1$ in such a way that it becomes the side of a simply connected embedded triangle $p'_1 p_2 p_1$. Note that this is possible because by the length conditions, $\gamma_1$ is (much) longer than $\ell_0$ while side $p_1 p'_1$ of the triangle is shorter than $\ell_0$. Furthermore, since the angle sum of a hyperbolic triangle is $<\pi$, the inserted arc and $\gamma_1$ form an angle $\theta'_1 > \theta_1 \geq \frac{\pi}{2}$ at $p'_1$.

\begin{figure}[h]
\begin{center}
\leavevmode \SetLabels
\L(.42*.64) $x$\\%
\L(.655*.488) $\ell_0$\\%
\L(.880*.67) $\alpha$\\%
\L(.891*.54) $\beta$\\%
\L(.42*-.02) (a)\\%
\L(.65*-.02) (b)\\%
\L(.88*-.02) (c)\\%
\L(.633*.888) $p_1$\\
\L(.672*.888) $p_2$\\
\L(.605*.450) $p'_1$\\
\L(.905*.89) $p_2$\\
\L(.622*.26) $\theta'_1$\\
\L(.088*.40) $\Gamma$\\
\L(.555*.45) $\gamma_1$\\
\L(.758*.70) $\gamma_2$\\
\endSetLabels
\AffixLabels{
\epsfig{file =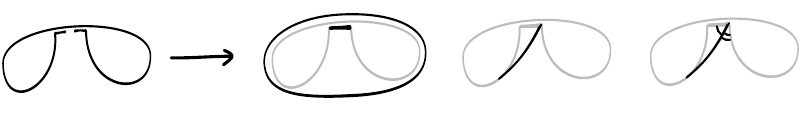,width=15.5cm,angle=0}}
\vspace{-24pt}
\end{center}
\caption{Inserting an arc of length $\ell_0$} \label{fig:sticktrick}
\end{figure}

The additional arc splits the angle $\theta_2$ at $p_2$ into two angles labelled $\alpha$ and $\beta$, where $\alpha$ is the interior angle of the triangle at vertex at $p_2$. As they sum up to at least $\frac{\pi}{2}$, at least one of them has to be at least $\frac{\pi}{4}$.

If $\beta \geq \frac{\pi}{4}$, then the target lamination is the exterior simple closed geodesic and we use the same argument as before, with $\beta$ in the role of $\theta_2$, arc $p_2 p'_1$ in the role of $x$ and $\theta'_1$ in the role of $\theta_1$.

If $\alpha\geq \frac{\pi}{4}$, then $\theta_1 < \pi - \alpha$ and loop $\gamma_1$ forms an angle $\alpha_1 > \alpha \geq \frac{\pi}{4}$ at $p_1$. In this case we do not make use of the inserted arc; instead, we apply the loop trick again. The target lamination is indicated in Figure \ref{fig:targetcaseIii}: 

\begin{figure}[h]
\begin{center}
\leavevmode \SetLabels
\L(.14*.65) $\alpha_1$\\%
\endSetLabels
\AffixLabels{
\epsfig{file =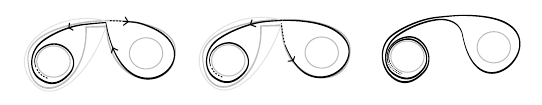,width=16cm,angle=0}}
\vspace{-24pt}
\end{center}
\caption{The two directions and the resulting target lamination for a subcase of Case I(ii). The lamination is also the target lamination for Case I(iv).} \label{fig:targetcaseIii}
\end{figure}

This concludes Case I(ii).

{\it Case I(iii) and (iv).} By symmetry, the two cases are identical and we consider Case (iv). The target lamination is the same as the one indicated on the right side of Figure \ref{fig:targetcaseIii}. As for the angle-segment-length condition: Let again $x$ be the length of the arc on $\Gamma$ between $p_1$ and $p_2$. If $x + \ell(\gamma_1) \geq \ell_0$ no insert is needed. Otherwise, $\gamma_2$ is the longer of the two loops and we insert an arc of length $\ell_0$ going from $p_1$ to its endpoint $p'_2$ on $\gamma_2$ so as to form an embedded triangle $p_1 p'_2 p_2$ and we conclude as in the preceding case.

This concludes Case I.

\subsubsection{Case II}

Again we have four sub-cases to consider, again labelled from (i) to (iv). As before the last two cases are the same by symmetry so we only need to deal with Cases (i) to (iii) as in Figure \ref{fig:seecase02}. 

Note the following difference however: the insert with a single dot here symbolizes an angle of at least $\frac{\pi}{4}$, and the double dot is an angle of at least  $\frac{3\pi}{4}$. As before, the cases overlap, but it only matters that they are exhaustive and we will construct a lamination for each case. 

\begin{figure}[h]
\begin{center}
\leavevmode
\SetLabels
\L(.335*-.05) (i)\\%
\L(.515*-.05) (ii)\\%
\L(.69*-.05) (iii)\\%
\L(.875*-.05) (iv)\\%
\L(.50*.50) $p_1$\\%
\L(.545*.38) $p_2$\\%
\L(.555*.64) $\Gamma$\\%
\endSetLabels
\AffixLabels{
\epsfig{file =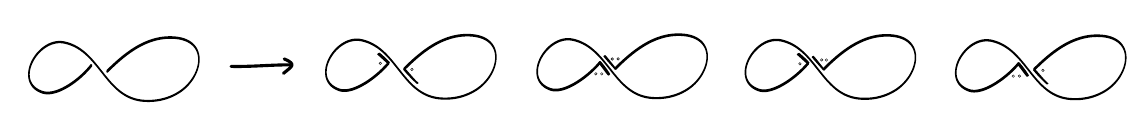,width=16cm,angle=0}}
\vspace{-12pt}
\end{center}
\caption{Case II} \label{fig:seecase02}
\end{figure}

In Case~II(i) (Figure~\ref{fig:seecase02}), we use the \emph{double loop trick}:
we apply the loop trick at both ends and wrap around both core geodesics in the direction of the dotted angles, This produces an embedded broken geodesic $\sigma$ with singular angles at least $\ge \pi/8$.

For Case~II(ii)(Figure~\ref{fig:seecase02}), we use a new trick that we call the {\it angle trick}. Note that $\Gamma$ and the exterior simple closed geodesic $c_3$ of the pair of pants with spine $\Gamma$ bound an annulus. This annulus is shown ``inside out'' in Figure \ref{fig:angletrick}.

{\it The angle trick.} Consider the annulus bounded by a simple geodesic and a broken geodesic as in Figure \ref{fig:angletrick}. Denote by $\alpha,\beta$ and $\vartheta_1$ the interior angles created by a geodesic loop at $p_1$. Since $\alpha$ and $\beta$ are angles of a triangle and the third angle is at least $\frac{3\pi}{4}$, it follows that $\alpha + \beta<\frac{\pi}{4}$ and hence $\vartheta_1 >\frac{\pi}{2}$. A similar argument holds with a loop of angle $\vartheta_2 >\frac{\pi}{2}$ at $p_2$.

\begin{figure}[h]
\begin{center}
\SetLabels
\L(.23*1.00) $p_1$\\%
\L(.23*-0.04) $p_2$\\%
\L(.618*.76) $\alpha$\\%
\L(.835*.76) $\beta$\\%
\L(.715*.77) $\vartheta_1$\\%
\L(.72*1.00) $p_1$\\%
\L(.336*.15) $\Gamma$\\%
\L(.19*.23) $c_3$\\%
\endSetLabels
\leavevmode
\AffixLabels{
\epsfig{file =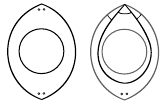,width=5cm,angle=0}}
\vspace{-12pt}
\end{center}
\caption{The angle trick} \label{fig:angletrick}
\end{figure}

Applying the loop trick we find a simple geodesic ray $\eta_1$ with initial point $p_1$ forming an angle $\vartheta'_1 > \frac{1}{2} \vartheta_1 \geq \frac{\pi}{8}$ with $\Gamma$ that spirals towards the exterior simple closed geodesic $c_3$ of the pair of pants as shown in Figure \ref{fig:yingyang}. A similar ray $\eta_2$ emanates under an angle $\vartheta'_2 > \frac{\pi}{8}$ from $p_2$ spiraling in the same sense as $\eta_1$. The two, together with $\Gamma$ form the homotopy class of the target lamination indicated on the right side of Figure \ref{fig:yingyang}.

\begin{figure}[h]
\begin{center}
\SetLabels
\L(.13*.45) $p_1$\\%
\L(.13*.60) $\alpha$\\%
\L(.10*.60) $\beta$\\%
\L(.114*.70) $\vartheta_1$\\%
\L(.375*.62) $\vartheta'_1$\\%
\L(.382*.45) $p_1$\\%
\L(.575*.539) $p_2$\\%
\L(.584*.36) $\vartheta'_2$\\%
\L(.46*.89) $\eta_1$\\%
\L(.69*.89) $\eta_2$\\%
\endSetLabels
\leavevmode
\AffixLabels{
\epsfig{file =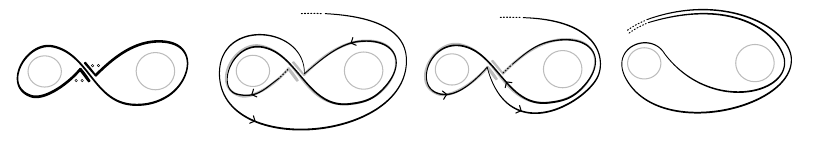,width=16cm,angle=0}}
\vspace{-24pt}
\end{center}
\caption{The lamination for Case II(ii)} \label{fig:yingyang}
\end{figure}

We are now in the final case, Case II(iii), where we use the {\it insertion trick} again if necessary. 

Consider the bold arc of length labelled $x$ in Figure \ref{fig:targetIIiii}. If $x\geq \ell_0$, we don't modify anything, but if not, then we insert an additional geodesic arc of length $\ell_0$ as in the upper right of Figure \ref{fig:targetIIiii}.

\begin{figure}[h]
\begin{center}
\leavevmode \SetLabels
\L(.38*.87) $x$\\%
\L(.82*.87) $\ell_0$\\%
\L(.915*.94) $\theta'_1$\\%
\L(.80*.635) $p_1$\\%
\endSetLabels
\AffixLabels{
\epsfig{file =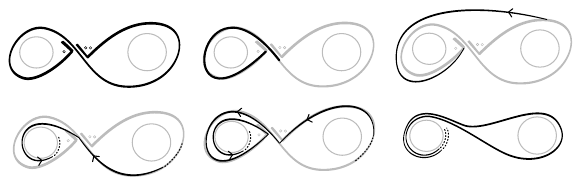,width=14cm,angle=0}}
\vspace{-24pt}
\end{center}
\caption{Construction of the lamination for Case II(iii)} \label{fig:targetIIiii}
\end{figure}

In both cases, we loop in both directions around the interior loop with target lamination as in the bottom right of figure \ref{fig:targetIIiii}. It is again easy to check that the broken geodesic satisfies the angle segment-length condition needed for Lemma \ref{lem:outside}. This concludes Case II.

\subsubsection{Case IV}

We now consider Case~IV from Figure~\ref{fig:seetopo}. The different subcases are portrayed in Figure \ref{fig:seecase04}. Here too, Cases (iii) and (iv) are equivalent and we will discuss Case (iii)

\begin{figure}[h]
\begin{center}
\leavevmode
\SetLabels
\L(.375*-.05) (i)\\%
\L(.553*-.05) (ii)\\%
\L(.725*-.05) (iii)\\%
\L(.90*-.05) (iv)\\%
\endSetLabels
\AffixLabels{
\epsfig{file =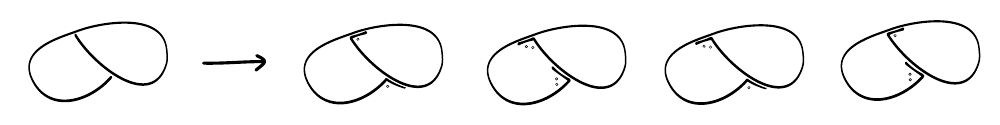,width=15.5cm,angle=0}}
\vspace{-12pt}
\end{center}
\caption{The subcases of Case IV} \label{fig:seecase04}
\end{figure}

For Case IV(i), we use a double loop trick with one interior loop and one exterior loop.

For Case IV(ii), we argue as in Case II(ii) and III(ii) using a combination of the angle trick and the loop trick. The key point is to observe what happens in the two geodesic interior loops based at the singular points. The complement of the geodesic loop is a triangle, one of whose angles is at least $\frac{3\pi}{4}$, hence the two remaining angles sum up to less than $\frac{\pi}{4}$. Therefore, the base angles for the geodesic loops  are both larger than $ \frac{3\pi}{4} - \frac{\pi}{4} = \frac{\pi}{2} $ (see Figure \ref{fig:IViiangle}). This being the case, the spiraling geodesics from the base points, as shown on the left and in the middle of Figure \ref{fig:IVii}  both form an angle of at least $\frac{\pi}{4}$.

\begin{figure}[th]
\begin{center}
\leavevmode \SetLabels
\L(.15*.92) $p_1$\\%
\L(.80*.15) $p_2$\\%
\endSetLabels
\AffixLabels{
\epsfig{file =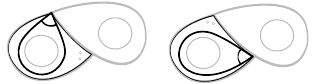,width=8cm,angle=0}}
\vspace{-24pt}
\end{center}
\caption{Both angles are greater than $\frac{\pi}{2}$} \label{fig:IViiangle}
\end{figure}

Using them we can portray the target lamination on the right.

\begin{figure}[th]
\begin{center}
\leavevmode
\AffixLabels{
\epsfig{file =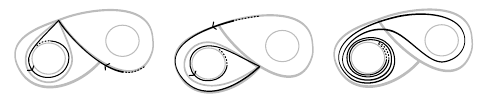,width=14cm,angle=0}}
\vspace{-12pt}
\end{center}
\caption{Case IV(ii)} \label{fig:IVii}
\end{figure}

Again, the resulting broken geodesic which represents the target lamination satisfies the angle-segment-length condition needed for Lemma \ref{lem:outside}.

In Case IV(iii), we argue similarly to Cases II(iii) and III(ii), using the insertion trick if necessary. If the segment labeled by $x$ in Figure \ref{fig:IViiistick} satisfies $x < \ell_0$, then we insert a segment. Note that the segment might not be embedded. What matters is that it is used to create a homotopy class of a target lamination. 

\begin{figure}[h]
\begin{center}
\leavevmode
\SetLabels
\L(.005*.30) $x$\\%
\L(.43*.637) $\ell_0$\\%
\L(.75*.635) $\ell_0$\\%
\L(.09*.85) $p_2$\\%
\L(.431*.85) $p_2$\\%
\L(.754*.85) $p_2$\\%
\L(.193*.28) $p_1$\\%
\L(.62*.25) $p'_2$\\%
\L(.84*.54) $p'_2$\\%
\endSetLabels
\AffixLabels{
\epsfig{file =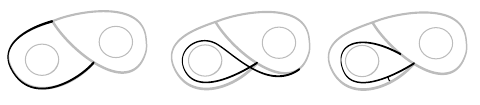,width=14cm,angle=0}}
\vspace{-36pt}
\end{center}
\caption{Case IV(iii)} \label{fig:IViiistick}
\end{figure}

In either case (that is, independently of the length of $x$), the target lamination is the same, and it is portrayed in Figure \ref{fig:IViii}. This completes Case IV.
\begin{figure}[h]
\begin{center}
\leavevmode
\AffixLabels{
\epsfig{file =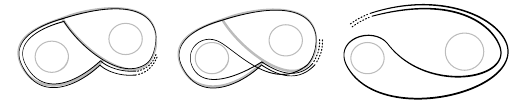,width=14cm,angle=0}}
\vspace{-36pt}
\end{center}
\caption{Case IV(iii)} \label{fig:IViii}
\end{figure}

\subsubsection{Case V}

Finally, we consider Case~V from Figure~\ref{fig:seetopo}. This is the only instance where the initial spine is the spine of a one-holed torus, which makes it seemingly more difficult than the previous cases. Nonetheless, the ideas are all the same, and the only real difficulty is visualizing what is going on. 
\begin{figure}[h]
\begin{center}
\leavevmode
\SetLabels
\L(.37*-.1) $(i)$\\%
\L(.54*-.1) $(ii)$\\%
\L(.703*-.1) $(iii)$\\%
\L(.87*-.1) $(iv)$\\%
\endSetLabels
\AffixLabels{
\epsfig{file =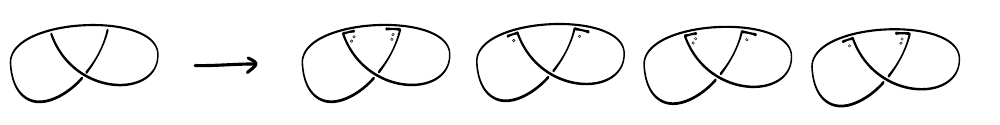,width=16cm,angle=0}}
\vspace{-24pt}
\end{center}
\caption{The subcases of Case V} \label{fig:seecase05}
\end{figure}

We first consider Case V(i). We consider the segment labelled $x$ in Figure \ref{fig:Vi}.

\begin{figure}[h]
\begin{center}
\leavevmode
\SetLabels
\L(.10*.7) $x$\\%
\L(.055*.90) $p_2$\\%
\L(.14*.90) $p_1$\\%
\L(.00*.20) $\gamma_1$\\%
\L(.17*.20) $\gamma_2$\\%
\L(.62*.25) $\theta'_1$\\%
\L(.465*.67) $\theta_2$\\%
\L(.5*.693) $\theta_1$\\%
\L(.667*.635) $\theta'_2$\\%
\L(.63*-.01) $p'_1$\\%
\L(.64*.90) $p_2$\\%
\L(.727*.90) $p_1$\\%
\L(.716*.55) $\ell'_0$\\%
\endSetLabels
\AffixLabels{
\epsfig{file =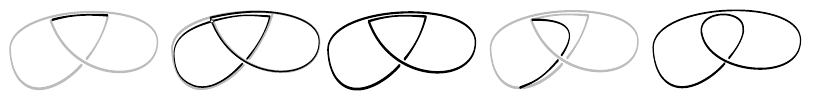,width=15cm,angle=0}}
\vspace{-24pt}
\end{center}
\caption{Case V(i)} \label{fig:Vi}
\end{figure}

One of the two loops labelled $\gamma_1$ and $\gamma_2$ must be longer. Since the configuration is symmetric we may assume w.l.o.g.\  that $\ell(\gamma_1) \geq \ell(\gamma_2)$. Now if $x \leq \ell_0$, we insert an arc $p_2 p'_1$ of length $\ell_0$ so as to obtain, as in earlier cases,  an embedded triangle $p_2 p'_1 p_1$ with side $p'_1 p_1$ of length $\ell'_0 \leq \ell_0$ on $\Gamma$ (Figure \ref{fig:Vi}). In either case, the target lamination is the simple closed geodesic as indicated on the right side of the figure. As for the angle-segment-length condition: since $\theta_1, \theta_2 \geq \frac{3\pi}{4}$ and the angle sum of the triangle is $<\pi$ we have $\theta'_2 \geq \frac{\pi}{2}$.

For Case V(ii) we make use of the homotopy class of the simple closed curve portrayed in Figure \ref{fig:Viiloop}. Note that the loop based at point $p_2$ is indeed positioned relatively to $\Gamma$ like in the figure, meaning that it intersects $\Gamma$ only at $p_2$ and, in particular, that its angle at $p_2$ is at least $\frac{\pi}{4}$. The loop and the simple closed geodesic in its homotopy class bound and annulus shown on the right side in Figure \ref{fig:Viiloop}. 

\begin{figure}[h]
\begin{center}
\leavevmode
\SetLabels
\L(.10*.90) $p_2$\\%
\L(.20*.18) $\Gamma$\\%
\endSetLabels
\AffixLabels{
\epsfig{file =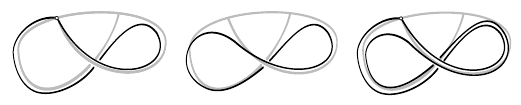,width=14cm,angle=0}}
\vspace{-24pt}
\end{center}
\caption{The loop and corresponding simple closed curve for Case V(ii). The two enclose an annulus shown on the right.}\label{fig:Viiloop}
\end{figure}

Now we make use of this curve to create, via the loop trick, a lamination by winding around it (in both directions) as in Figure \ref{fig:Vii}. 

\begin{figure}[h]
\begin{center}
\leavevmode
\SetLabels
\endSetLabels
\AffixLabels{
\epsfig{file =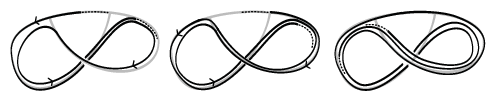,width=14cm,angle=0}}
\vspace{-24pt}
\end{center}
\caption{Case V(ii)} \label{fig:Vii}
\end{figure}

Our final case to consider is Case V(iii). Again, we argue in terms of $x$ from Figure \ref{fig:Viiilong}. 

\begin{figure}[h]
\begin{center}
\leavevmode
\SetLabels
\L(.225*.54) $x$\\%
\endSetLabels
\AffixLabels{
\epsfig{file =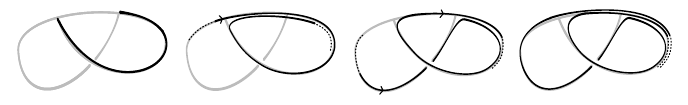,width=14cm,angle=0}}
\vspace{-24pt}
\end{center}
\caption{Case V(iii), when $x \geq \ell_0$} \label{fig:Viiilong}
\end{figure}

If $x \geq \ell_0$, then we construct the target lamination as in the figure. 

If $x < \ell_0$ and loop $\gamma_1$ forms an angle $\theta_1 \geq \frac{\pi}{2}$ at its vertex $p_1$ then we use as target lamination the simple closed geodesic $c_1$ (Figure \ref{fig:Viiishort} (d)) in the free homotopy class of $\gamma_1$.

Assume now that $x < \ell_0$ and $\theta_1 < \frac{\pi}{2}$. In this case we add a segment of length $\ell_0$ as in Figure \ref{fig:Viiishort} (b), (c) forming an embedded triangle $p_1 p'_1 p_2$,  again with side $p_1 p'_1$ on $\Gamma$ shorter than $\ell_0$.

\begin{figure}[h]
\begin{center}
\leavevmode
\SetLabels
\L(.226*.54) $x$\\%
\L(.025*.10) $\gamma_1$\\%
\L(.148*.735) $\theta_1$\\%
\L(.172*.95) $p_1$\\%
\L(.434*.595) $\ell_0$\\%
\L(.61*.62) $\alpha$\\%
\L(.13*-.08) (a)\\%
\L(.373*-.08) (b)\\%
\L(.617*-.08) (c)\\%
\L(.86*-.08) (d)\\%
\L(.55*.24) $p'_1$\\%
\L(.56*.01) $\theta'_1$\\%
\L(.655*.96) $p_1$\\%
\L(.56*.91) $p_2$\\%
\L(.78*.30) $c_1$\\%
\L(.84*.17) $\gamma_1$\\%
\L(.89*.96) $p_1$\\%
\endSetLabels
\AffixLabels{
\epsfig{file =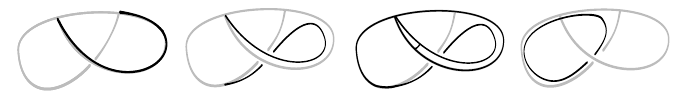,width=15cm,angle=0}}
\vspace{-12pt}
\end{center}
\caption{Case V(iii), when $x < \ell_0$} \label{fig:Viiishort}
\end{figure}
This case is the reason why we imposed an angle of at least $\frac{3\pi}{4}$. We consider the angle $\alpha$ at vertex $p_2$ of the triangle as in the figure. If $\alpha \geq \frac{\pi}{4}$, then we use as target lamination again the simple closed geodesic $c_1$ (Figure \ref{fig:Viiishort} (d)) which is also homotopic to the loop portrayed in (c) that goes from $p_2$ along the inserted segment to $p'_1$ and from there along $\Gamma$ back to $p_2$.

\begin{figure}[h]
\vspace{12pt}
\begin{center}
\leavevmode
\SetLabels
\L(.27*.7) $\beta$\\%
\L(.35*.785) $\beta'$\\%
\L(.205*.6) $\alpha$\\%
\L(.16*1.02) $p_2$\\%
\L(.85*0.80) $\gamma_2$\\%
\L(.70*0.54) $c_2$\\%
\endSetLabels
\AffixLabels{
\epsfig{file =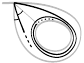,width=4.5cm,angle=0}}
\vspace{-12pt}
\end{center}
\caption{Case V(iii), when $\alpha < \frac{\pi}{4}$} \label{fig:angleV}
\end{figure}

If,  however, $\alpha < \frac{\pi}{4}$ we use a version of the loop trick in the annulus bordered by $\gamma_2$ and the simple closed geodesic $c_2$ in its free homotopy class, but where we have to pay extra attention to angles. Consider Figure \ref{fig:angleV}. The inserted arc $p'_1p_2$
makes an angle of at least $\frac{\pi}{8}$ with the perpendicular from $p_2$ to $c_2$ (in the figure the angle is marked $\beta$). This is because
\begin{equation*}
\beta \geq  \frac{\frac{3\pi}{4}}{2} -\alpha > \frac{\pi}{8}.
\end{equation*}
The inserted segment $p_2p'_1$ at angle $\alpha$ shown in Figure \ref{fig:angleV} intersects $c_0$. By symmetry, the spiraling geodesic ray on the right hand side makes therefore an angle $\beta' > \beta \geq \frac{\pi}{8}$ with the perpendicular. We can now attach it to the segment $p_2p'_1$ so as to get the homotopy class of the target lamination which is the same as in Case V(iii) when $x \geq \ell_0$ (Figure \ref{fig:Viiilong}). The lower bound on $\beta$ allows us to satisfy the angle-segment-length condition needed for Lemma~\ref{lem:broken} with this lamination.

This completes the proof in Case~V and hence the proof of Theorem~\ref{thm:main}.

\section{Corollaries}
Here we deduce two results from our main theorem.

\subsection{Measuring the Birman-Series set}

The set of all simple and complete geodesics is a nowhere dense set on a hyperbolic surface, and previous results \cite{Buser-Parlier} have shown that the non-density can be quantified in terms of topology. Here we use our main theorem to prove the following, which uses the existence of surfaces with long systoles, first exhibited by Buser-Sarnak \cite{Buser-Sarnak}. 

\begin{corollary}\label{cor:dense}
There exists a family of closed hyperbolic surfaces, $X_{g_k}$, with a $\delta(g_k)$-dense Birman-Series set with 
$$
\delta(g_k) = e^{10}\, g_k^{-2/3}.
$$
\end{corollary}
\begin{proof}
There are by now many constructions of surfaces with $4/3 \log(g)$ systole growth. Katz--Schaps--Vishne \cite{KSV} proved that principal congruence Hurwitz surfaces, for a sequence of genera $g_k \to \infty$ satisfy, for $g_k\geq 65$ that 
$$
\sys(X_{g_k}) \geq \frac{4}{3} \log(g_k). 
$$
Now on these surfaces, since every point has injectivity radius at least $\sys / 2$, this means every point $p$ on the surface is the midpoint of an embedded simple geodesic arc of length $\frac{4}{3} \log(g_k)$. 

We can now apply the main theorem which says that $p$ is distance at most $\varepsilon$ from a geodesic lamination where $\varepsilon$ satisfies 
\[
 \frac{4}{3} \log(g_k) \leq 2 \log\left(\frac{1}{\varepsilon} \right) +20
\]
and hence
\[
\varepsilon \le e^{10}\, g_k^{-2/3}.
\]
\end{proof}
\begin{remark}
Note that any construction which uses surfaces of systole roughly $4/3 \log(g)$ will result in a density bound on the order of $g^{-2/3}$. These families of surfaces are only known to exist for certain genera. It is known that there do exist families of surfaces in any genus with logarithmic growth, but the constant in front is no longer $4/3$. More precisely, Katz--Sabourau \cite{Katz-Sabourau} constructed families with $19/120 \log(g)$ growth and Liu--Petri \cite{Liu-Petri} exhibited the existence of surfaces with growth on the order of $2/9 \log(g)$ via random methods. The Liu-Petri result implies the existence of families in every genus with density on the order of $g^{-1/9}$. 
\end{remark}

\subsection{Drawing the Birman-Series set}

This second consequence will be a bit more descriptive, but the general idea is the following. In a previous paper, we provided illustrations of the Birman-Series set \cite{Buser-Parlier}. To make these required checking the simplicity of geodesic arcs for arcs of sufficiently long length. As in all things algorithmic, this can only be checked for finite length arcs. For a genuine lamination, any of its finite length subarcs will be simple. It follows from standard hyperbolic geometry, that the full Birman-Series set passing through a given region lies in a small neighborhood of the set of all finite length simple arcs of sufficiently long length that also pass through that same region. And in fact, this is the spirit of the original proof of this phenomenon \cite{Birman-Series}. What isn't immediate is that a finite length simple arc actually corresponds to a genuine lamination. Theorem \ref{thm:main} provides that missing piece: if a geodesic arc of sufficiently long length is embedded, then its proximity at its core to a geodesic lamination is guaranteed uniquely by length of its "tails".

{\it Addresses:}\\
EPFL, Lausanne, Switzerland\\
University of Fribourg, Switzerland

{\it Emails:}\\
peter.buser@epfl.ch\\
hugo.parlier@unifr.ch

\end{document}